\documentclass[11pt,twoside]{article}
\usepackage{latexsym}
\usepackage{amssymb,amsbsy,amsmath,amsfonts,amssymb,amscd,amsthm}
\usepackage{epsfig,graphicx,mathrsfs}
\usepackage{color}
\usepackage{enumerate}
\usepackage{pdfsync}
\usepackage{url}
\newcommand{\R}{\mathbb{R}}

\newcommand{\N}{\mathbb{N}}

\newcommand{\E}{\mathbb{E}}

\newcommand{\mJ}{\mathcal J}

\newcommand{\beq}{\begin{equation}}
\newcommand{\eeq}{\end{equation}}
\def\a{\alpha}
\def\b{\beta}

\def\g{\gamma}

\def\G{\Gamma}

\def\t{\tau}

\def\pd{\partial}
\def\half{\frac{1}{2}}

\newcommand{\cA}{{\cal A}}

\newcommand{\cC}{{\cal C}}

\newcommand{\cF}{{\cal F}}

\newcommand{\cH}{{\cal H}}

\newcommand{\cM}{{\cal M}}

\newcommand{\cP}{{\cal P}}

\newcommand{\fDT}{D_{[t,T)}^{1-\beta}}
\newcommand{\fDz}{D_{(0,t]}^{1-\beta}}
\newcommand{\fdT}{\partial_{[t,T)}^{\beta}}
\newcommand{\fdz}{\partial_{(0,t]}^{\beta}}
\newcommand{\fIT}{I_{[t,T)}^{\beta}}
\newcommand{\fIz}{I_{(0,t]}^{\beta}}

\newcommand{\diver}{{\rm div}}
\newtheorem{theorem}{Theorem}[section]
\newtheorem{lemma}[theorem]{Lemma}
\newtheorem{definition}[theorem]{Definition}
\newtheorem{proposition}[theorem]{Proposition}

\newtheorem{remark}[theorem]{Remark}

\numberwithin{equation}{section}
\begin{document}
\title{\Large \bf A time-fractional mean field game}
\author{Fabio Camilli\footnotemark[1] \and Raul De Maio\footnotemark[1]}
\date{\today}
\maketitle
\begin{abstract}
We consider a   Mean Field Games model  where the dynamics of the agents
is subdiffusive. According to the optimal control interpretation of the problem, we get a
system involving fractional time-derivatives  for the Hamilton-Jacobi-Bellman and the Fokker-Planck  equations.
We discuss separately the well-posedness for each of the two equations and
then  we prove existence and uniqueness of the solution to the Mean Field Games system.
\end{abstract}
 \begin{description}
    \item [{\bf AMS subject classification}:] 35R11, 60H05, 26A33, 40L20, 49N70.
     \item [ {\bf Keywords}:] subdiffusion; time change; fractional derivate; fractional Fokker-Planck equation;  fractional Hamilton-Jacobi-Bellman equation; Mean Field Games.
\end{description}
%
\footnotetext[1]{Dip. di Scienze di Base e Applicate per l'Ingegneria,  ``Sapienza'' Universit{\`a}  di Roma, via Scarpa 16,
00161 Roma, Italy, ({\tt e-mail:camilli,demaio@sbai.uniroma1.it})}
%
\pagestyle{plain}
\pagenumbering{arabic}
\section{Introduction}
The study of complex systems and the investigation of their structural and dynamical properties   is at forefront of the interaction
between mathematics and real life sciences such as biology, sociology, epidemiology, etc. Complex systems are characterized by a large number of elementary individuals and strong interactions  among these individuals  which make their evolution hardly predictable by traditional approaches. Hence  in the recent years new techniques have been developed  in order to capture specific properties of these systems.\par
 The Mean Field Games (MFG in short) theory introduced in \cite{hmc}, \cite{ll}   (see  \cite{gs} for a review) is a new paradigm    in the framework of dynamics games  with a large number of players  to deal
  with systems composed by   rational individuals,   i.e. individuals able to choose their behavior on the basis of a  set of preferences and  to change it  in consequence of the interaction with other   members of the population.  MFG theory has been very successful in applications where an intrinsic rationality is embedded in the complex system such as  pedestrian motion, opinion formation, financial market, management of exhaustible resources, etc.  From a mathematical point of view a  MFG model leads to the study of a strongly coupled system
of two partial differential equations: a Hamilton-Jacobi-Bellman (HJB in short) equation related to the optimal control problem solved by a single agent, but influenced by the presence of all (or a part of) the other agents;  a  Fokker-Planck (FP in short)  equation governing the evolution of the distribution of the population and  driven by the optimal strategies chosen by the agents. We refer to \cite{carda} for a nice introduction to the MFG theory.\par
In the  model introduced in \cite{hmc}, \cite{ll}   the dynamics of the single agent is governed by a (possibly degenerate) Gaussian diffusion process. Hence  the  underlying environment has no role in the  problem or, in other words,   is isotropic, an assumption  not satisfied in several applications.\\
The study of anomalous diffusion processes   deviating from the classical diffusive behavior
has lead to the introduction of a class of subdiffusive processes  which displays local motion occasionally interrupted by long sojourns, a trapping effects due to the anisotropy of the medium.  A subdiffusive regime is considered to be a  better model not only for several transport phenomena in physics, but also, for example, in the study of   volatility of  financial markets, bacterial motion, bird flight, etc. (see \cite{mk} for a review). It is worth noting that the  nonmarkovian nature of these processes can be also interpreted  as a manifestation of   memory effects encountered in the study of the phenomena. \\
At a microscopic level,  the classical construction of   a diffusion process as  the
limit of a  Markov chain with finite time average/finite jump variance is replaced, for a subdiffusive process, by the limit of a Continuous Time Random Walk displaying     broad spatial jump distribution  and/or infinite waiting time  between consecutive jumps; at a macroscopic one, the evolution of the probability density function (PDF) associated to a subdiffusive process is governed by a time-fractional FP equation, i.e. a FP equation involving fractional derivatives in time.\par
In this paper, we   introduce a class of  MFG problems in which the dynamics of the agent is subdiffusive rather diffusive as in  the  Lasry-Lions model. The aim is to present a simplified framework in which we take into account only   waiting time effects and not jump distribution ones since this model already presents several interesting  differences compared to the classical MFG theory.  A first important point is
to understand the correct formulation of the MFG system in this framework.
Time-fractional   FP equation governing the evolution of the  PDF of a subdiffusive process were first derived  for the case of a space-dependent drift and constant diffusion coefficient \cite{ms2}. Since then, the theory  has been progressively
generalized to include the case of space-time dependent coefficients which is relevant for our study (see \cite{mgz}, \cite{nane}).
We will consider     weak solution based on a measure theoretic approach, see  \cite{bs},  and we will prove some regularity properties of the  measure-valued solution. \par
On the other hand, a theory of viscosity solution for time-fractional HJB equations has been developed
in \cite{gn} for the first order case and in \cite{ty}  for the second order one (including also  a nonlocal operator in the space variable). We remark that, in these papers, the connection between the HJB equation and the associated control problem via the dynamic
programming principle is not considered. Here we will exploit the Ito's formula  and the corresponding properties developed in \cite{kobayashi} to establish this connection; indeed, this is a crucial point to understand the differential game   associated to the MFG system and to  justify  the specific  form of the  HJB equation considered.
In view of the well-posedness of the time-fractional MFG system  in  Section \ref{sec_MFG}, which requires
a   regular setting for the notion of   solution to the HJB equation, we consider the  theory of {\it mild solutions}   introduced in \cite{kv1} which gives, under appropriate assumptions,   existence and uniqueness of a classical solution.\par
After having considered  the two equations separately, we tackle  the study of the MFG system
\[
\left\{
\begin{array}{lll}
-\partial_t v +\fDT[-\nu \Delta v  + \cH(t,x, D  v) -    G(x, m)]=0,& (t,x) \in (0,T) \times \R^d\\[4pt]
\partial_t m + \left[\nu \Delta \cdot+ \diver \big(D_p \cH(t,x, D  v)\cdot\big)\right] (\fDz m) = 0,\\
m(0,x) = m_0(x) ,\quad v(T,x) = g(x).
\end{array}
\right.\]
where $\fDT$ and $\fDz$ denote the backward and  forward Riemann-Liouville fractional derivatives,
$\cH$ is a scalar Hamiltonian,   convex in the gradient variable, and the
term $G$  associates a real valued function $G(x,m)$   to a probability density $m$.
 We obtain   existence of a
solution by a fixed point argument; moreover we   get uniqueness of the solution
adapting a  classical argument in MFG theory  to this framework and taking advantage of the duality relation
between  the  two equations composing the system.  \par
We   mention that the theory of linear and semilinear fractional  differential equations
is currently an active field of research  \cite{acv1,ly,z}. For fully nonlinear equations,
the theory is at the beginning and, besides the papers \cite{gn,kv2,ty} dealing with  the  HJB equation,
we also mention \cite{acv2} where a    fractional porous media equation is studied.\par
The paper is organized as follows. In Section \ref{sec_frac}, we review some basic definitions of the fractional calculus.
In Section \ref{sec_FP} we introduce the class of subdiffusive processes and we
study the corresponding FP equation.  Section \ref{sec_HJB} is devoted to the time-fractional
HJB equation. Finally, in Section \ref{sec_MFG}, we prove the well posedness of the time-fractional MFG system.

\section{Fractional calculus}\label{sec_frac}
In this section, we   remind the reader     definitions and some basic   properties of
fractional   operators (we refer to    \cite{po} for a
 complete account of the theory). \\
Throughout this section, we always assume that $\b\in (0,1]$.
For   $f:(a,b) \to \R$, the forward and backward Riemann-Liouville fractional integrals are defined by
\begin{align*}
I^{\b}_{(a,t]} f(t):=\frac{1}{\G(\b)}\int_a^t f(\t)\frac{1}{(t-\t)^{1-\b}}d\t, \\
I^{\b}_{[t,b)} f(t):=\frac{1}{\G(\b)}\int_t^b f(\t)\frac{1}{(\t-t)^{1-\b}}d\t.
\end{align*}
The fractional integrals are bounded linear operator over $L^p(a,b)$, $p\ge 1$; indeed, by Holder's inequality,   if $f \in L^p(a,b)$, then
\[
\|I^\b_{(a,t]} f\|_{L^p} \leq \frac{|b - a|^{ \b /p}}{\b\G(\b)}\|f\|_{L^p}.
\]
The forward  Riemann-Liouville  and Caputo derivatives are defined, respectively, by
  \begin{align}
 D^{\b}_{(a,t]}  f(t):=\frac{d}{dt} \left[I^{1-\b}_{(a,t]} f(t) \right]=\frac{1}{\G(1-\b)}\frac{d}{dt} \int_a^t f (\t)\frac{1}{(t-\t)^\b}d\t,\label{fRL}\\
  \pd^{\b}_{(a,t]} f(t):=I^{1-\b}_{(a,t]} \left[\frac{df}{dt}(t)\right] =\frac{1}{\G(1-\b)}\int_a^t\frac{df}{dt}(\t)\frac{1}{(t-\t)^\b}d\t.\label{fC}
  \end{align}
while the   backward Riemann-Liouville and Caputo   derivatives are defined, respectively, by
\begin{align}
  & D^{\b}_{[t,b)} f(t):=-\frac{d}{dt} \left[I^{1-\b}_{[t,b)} f(t) \right]=-\frac{1}{\G(1-\b)}\frac{d}{dt} \int_t^b f (\t)\frac{1}{(\t-t)^\b}d\t,\label{bRL}   \\
  & \pd^{\b}_{[t,b)} f(t):=-I^{1-\b}_{[t,b)} \left[\frac{df}{dt}(t)\right]=-\frac{1}{\G(1-\b)}\int_t^b \frac{df}{dt}(\t)\frac{1}{(\t- t)^\b}d\t. \label{bC}
\end{align}
It is easy  that for $\b\to 1$ the  Riemann-Liouville and Caputo fractional derivatives of a smooth function $f$
converge to the classical derivative $\frac{df}{dt}$, i.e. fractional derivatives are an extension of standard derivatives.
But, since fractional derivatives are defined by an integral and therefore are nonlocal operators,
several properties of differential calculus do  not hold. For example, the product formula
$\pd^{ \b}_{(a,t]} (fg)=(\pd^{ \b}_{(a,t]}  f)g+(\pd^{ \b}_{(a,t]} g)f$ does not hold and consequently there is no
useful formula for the integration by parts. Moreover, if $\alpha+\b<1$, the identity $\pd^{\alpha+\b}_{(a,t]} f=\pd^{\alpha}_{(a,t] }(\pd^{ \b}_{(a,t] } f)$
is in general false and it can be recovered only for $f$ smooth.\\
Note that  the Riemann-Liouville derivative is defined for  a function
$f\in C^0([0,T])$, while the Caputo derivative, even if it is a derivative of order less than $1$, is defined
for  a function $f\in C^1([0,T])$. If $f\in C^1([0,T])$, the following formulas, easily obtained by integration by parts,
establish the connection between the two types of fractional derivatives
\begin{align}
   \pd^{ \b}_{(a,t]} f(t)= D^{\b}_{(a,t]} f(t)-\frac{(t-a)^{-\b}}{\G(1-\b)}f(a),\\
   \label{B1}
   \pd^{\b}_{[t,b)} f(t)= D^{\b}_{[t,b)} f(t) - \frac{(b-t)^{-\b}}{\G(1-\b)}f(b).
\end{align}
The previous identities, also called \emph{regularized Caputo derivatives}, allow
to define Caputo derivatives under less stringent regularity assumptions for $f$.\\

\section{The time-fractional FP equation}\label{sec_FP}
This section is devoted to the study of the time-fractional FP equation, i.e. the FP equation
describing the evolution of the PDF  for a subdiffusive stochastic process. \\
A classical way to define a diffusion process is taking the limit in an appropriate sense of a random walk in dimension 1.
The corresponding  construction for a subdiffusive process   is performed through rescaled  limits of  a  continuous-time random
walk  (CTRW in short). In a CTRW model, a random waiting time $\g_i$ occurs between successive random   jumps  $\xi_i$.
Jumps and    waiting times form a sequence of  i.i.d., mutually independent   random variables $(\g_i)_{i\in\N}\subset \R^+$, $(\xi_i)_{i\in\N}\subset \R^d$.
Set  $s(0)=0$, $t(0)=0$ and let $s(n)=\sum_{i=1}^n \xi_i$ and $t(n)=\sum_{i=1}^n \g_i$ be the position
of the particle after $n$ jumps and the time of $n^{th}$ jump. For $t\ge 0$, define $n(t)=\max\{n\ge 0:\, t(n)\le t\}$,  the number
of jumps by time $t$, and observe that $n(t)$ and $t(n)$ have the inverse relationship
$\{n(t)\ge n\}=\{t(n)\le t\}$. The CTRW process
\[
 x(t)=\sum_{i=1}^{n(t)}\xi_i
\]
determines the location reached at time $t$, for a particle performing a random walk in which random particle jumps are separated by random waiting times. Note that the process $x(t)$ is not in general Markovian.\\
Consider  the limit of a CTRW process  for the following  standard   scaling  (see \cite{kv2,ms,ms2}):
$\xi_i\mapsto \t^{1/\a} \xi_i $, $\g_i\mapsto \t^{1/\b}\g_i $.
If the waiting times have finite mean and the jumps have finite variance, then the scaled CTRW converges in distribution
to a diffusion process. However, if the waiting times have infinite mean, the limit   process is the composition of a  $\a$-stable L\'evy motion $L(t)$ and the inverse of a $\b$-stable subordinator $E(t)$, where $\a\in (0,2]$ and $\b\in (0,1)$. In this paper we only consider
the case  $\a=2$ (the general  case will be considered elsewhere). In this case   the limit process is a time-changed diffusion process which is described by the following stochastic differential equation
\begin{equation*}
\left\{
\begin{array}{ll}
dX_t = b(t, X_t)dE_t + \sigma(t, X_t)dB_{E_t},\\
X_0 = x_0
\end{array}
\right.
\end{equation*}
where $B_t$ is a Brownian motion in $\R^p$ and
$E_t$ is  the inverse of a $\beta$-stable subordinator $D_t$  with Laplace  transform $\E[e^{-\tau D_t}]=e^{-t\cdot\tau^\beta}$, i.e.
\begin{equation}\label{subor}
E_t := \inf\{\tau>0 : D_\tau>t\},\hspace{2 mm}t\geq 0.
\end{equation}
The process $E_t$    is continuous and nondecreasing, moreover  for any $t, \gamma > 0$  its  $\gamma$-moment is given by
\begin{equation}\label{gmoment}
\mathbb{E}(E_t^\gamma) = C(\beta, \gamma)t^{\beta\gamma}
\end{equation}
 for some positive  constant $C(\beta, \gamma)$ (see \cite{ms2}). Note that  the process $E_t$ does not have stationary   and independent increments.\\
To obtain a more explicit description of the process $X_t$, consider  the diffusion process $Y_t$  given by
\begin{equation}\label{sdegno}
\left\{
\begin{array}{ll}
dY_t = b(D_t, Y_t)dt + \sigma(D_t, Y_t)dB_t,\\
Y_0 = x_0,\, D_0 = 0
\end{array}
\right.
\end{equation}
and assume that $b: [0,T]\times \R^d    \to     \R^d$, $\sigma:[0,T]\times  \R^d   \to  \R^{d\times p}$ are continuous and satisfy
\begin{align}
     |b(t,x) - b(t,y)| + |\sigma(t,x)- \sigma(t,y)| &\leq L |x - y|,\quad &\forall x,y\in\R^d, t\in [0,T]\label{hype0}\\
     |b(t,x)| + |\sigma(t,x)| &\leq   M,\quad &\forall x \in\R^d, t\in [0,T],\label{hype1}
\end{align}
for constants $L,M > 0.$
Then, the process $Y_t$ is well defined and the process $X_t$  given by \eqref{sdegno} can be represented  in the subordinated form
\begin{equation}\label{change}
   X_t=Y_{E_t}.
\end{equation}
 As explained in  \cite{kobayashi, mgz}, the formula \eqref{change} allows to obtain the following interpretation of the process $X_t$: the subordinator
 $E_t$ is a change of time, hence  the standard time $t$ can be interpreted as an external  scale,  or the time scale of an external observer, while $E_t$ as an internal   scale which introduces trapping events in the motion. Note that   the change of time induced by subordination  influences  only the dependence of the coefficients $b$, $\sigma$ on the time variable. Hence,  between two jumps   when the particle is not trapped,     the process moves according to a standard diffusion process $Y_t$ since it holds $ D_{E_t}=t$. \par
Formula \eqref{change} allows to deduce the time-fractional FP equation satisfied by the PDF of the process $X_t$
(see  \cite{bs,hku,mgz,nane}). We have
\begin{equation}\label{FPEfr}
\partial_t m(t,x) = \cA \left[\fDz m(t,x)\right],
\end{equation}
where     $\fDz$ is the forward Riemann-Liouville derivate of order $1-\beta$, see \eqref{fRL}, and
\begin{equation*}
\cA \cdot  = -\sum_{i=1}^d  \frac{\pd}{\pd x_i}[b_i(t,x)\cdot]+ \sum_{i,j=1}^d \frac{\pd^2}{\pd x_i\pd x_j}[ a_{ij}(t,x)\cdot ]
\end{equation*}
with $a=\half\sigma(t,x) \sigma(t,x)^t$.
For $\beta=1$,   equation \eqref{FPEfr} reduces to the classical FP equation since in this case  $\fDz m=m$. Moreover, if
the coefficients $b$, $\sigma$ are independent of $t$,  exchanging the derivates in time with the ones in space and applying the
fractional integral $\fIz$ on both the sides, \eqref{FPEfr} can be rewritten as
\[
\fdz m(t,x) = \cA  m(t,x),
 \]
where $\fdz  $ is the forward Caputo derivate of order $\b$, see \eqref{fC}. \\
To introduce a notion of weak solution for \eqref{FPEfr} we need the following result:
\begin{lemma}\label{byparts} Let   $u \in L^1([0,T]\times \R^d)$, then
\begin{align}
\langle D_{(0,t]}^{1-\beta}u, f \rangle = \int_0^T \int_{\R^d} u(t,x) D_{[t,T)}^{1-\beta}f(t,x)dxdt, \label{bypart1}\\
\langle D_{[t,T)}^{1-\beta}u, f \rangle = \int_0^T \int_{\R^d} u(t,x) D_{(0,t]}^{1-\beta}f(t,x)dxdt, \label{bypart2}
\end{align}
for every $f \in C^{\infty}_{c}((0,T)\times \R^d).$
\end{lemma}
\begin{proof}
Taking into account   the identity
$$\int_0^T k(t)\fIT h(t)  dt = \int_0^T h(t) \fIz k(t)dt,$$
we have    for $f \in C^{\infty}_c((0,T) \times \R^d)$
\begin{align}
&\int_0^T \int_{\R^d}\fDz u(t,x) f(t,x) dxdt =   \int_0^T \int_{\R^d}\partial_t(\fIz u)(t,x) f(t,x) dxdt  =\nonumber\\
&   \int_{\R^d}\fIz u(t,x)f(t,x)dx \mid_0^T - \int_0^T \int_{\R^d} \fIz u(t,x)\pd_tf(t,x)dx dt= \label{conto}\\
&- \int_0^T \int_{\R^d} u(t,x) \fIT(\pd_tf(t,x))dx dt= \int_0^T \int_{\R^d} u(t,x) \partial_{[t,T)}^{1-\beta} f(t,x)dx dt\nonumber
\end{align}
and the identity \eqref{bypart1} follows observing that for $f \in C^{\infty}_{c}((0,T)\times \R^d)$, $\partial_{[t,T)}^{1-\beta} f(t,x)=\fDT f(t,x)$.
The identity \eqref{bypart2} is proved in a similar way.
\end{proof}
The previous lemma, together with the usual distributional rules, justifies the following:
\begin{definition}\label{weak2}
Given $m_0\in \cP_1(\R^d)$,  $m\in L^1([0,T],\cP_1(\R^d))$ is said to be a weak solution to \eqref{FPEfr} with the initial condition $m(x,0)=m_0(x)$ if for any test function $\phi\in C^\infty_c([0,T)\times\R^d)$, we have
\[\int_{\R^d}\phi(0,x)dm_0(x)+\int_0^T\int_{\R^d}\Big[\pd_t\phi+\fDT\big( \sum_{i=1}^d  b_i(t,x)\frac{\pd\phi}{\pd x_i}   + \sum_{i,j=1}^d a_{ij}(t,x)\frac{\pd^2\phi}{\pd x_i\pd x_j}\big)   \Big]dm(t)(x)=0.\]
\end{definition}
The following result for   \eqref{FPEfr} has been proved in \cite{bs,mgz,nane}
\begin{theorem}\label{thm_FFP1}
If $m_0$ is the law of $X_0$, then the law $m_t$ of   $X_t $ is the unique weak solution of the fractional FP equation  \eqref{FPEfr}
such that $m(0) = m_0$.
\end{theorem}
To conclude this section we provide some regularity properties of the solution of \eqref{FPEfr}
(see \cite{carda} for a corresponding result in the classical case). Denote with $\cP_1(\R^d)$ the set of the probability measures over $\R^d$ with finite first moment,  endowed with the Kantorovich-Rubinstein distance
$$d_1(\mu_1, \mu_2) := \inf_{\g\in\Pi(\mu_1,\mu_2)}\int_{\R^{2d}}|x-y|d\g(x,y),
$$
where $\Pi(\mu_1,\mu_2)$ is the set of the Borel probability measures on $\R^{2d}$ with marginal distributions $\mu_1$, $\mu_2$.

\begin{proposition}
Given $m_t$   as in Theorem \ref{thm_FFP1}, then
\begin{itemize}
  \item  The map $m:[0,T] \to \cM^+(\R^d)$ is $\frac{\beta}{2}$-Holder continuous, i.e. there is a constant $C= C(M,L)$, where
  $M,L$ as in \eqref{hype0}-\eqref{hype1},  such that
  for every $t,s \in [0,T]$
\begin{equation}\label{est_m1}
d_1(m_t, m_s) \leq C |t - s|^{\beta/2},
\end{equation}
\item There is a constant $c=c(M)$ such that
\begin{equation}\label{est_m2}
 \int_{\R^d}|x|^2dm_t(x) \leq c\left(\E|X_0|^2 + t^{2\beta} + t^{\beta}\right)
\end{equation}
\end{itemize}
\end{proposition}
\begin{proof}
We first observe that
$$d_1(m(t),m(s))\le \int_{\R^{2d}}|x-y| d\pi(x,y)=\E[|X_t-X_s|]$$
where $\pi$ is the law of the pair $(X_t,X_s)$.
Assume w.l.o.g. $s<t$, then
\begin{align*}
X_t - X_s &=\int_s^t dX_z = \int_s^t b(X_z, z)dE_z + \int_s^t \sigma(X_z, z)dB_{E_z}.
\end{align*}
By \eqref{hype1}, we have
\begin{align*}
\mathbb{E}(|X_t - X_s|) &\leq M (\mathbb{E}(|E_t - E_s|) + \mathbb{E}(|B_{E_t} - B_{E_s}|)).
\end{align*}
Since $t\geq s$, then $E_t \geq E_s$ a.s.. Moreover,  by  \eqref{gmoment}, we have
\[
\mathbb{E}(|E_t - E_s|)  = \mathbb{E}(E_t - E_s) = C(\beta, 1)(t^{\beta} - s^{\beta})\leq C(\beta, 1)| t - s |^{\beta}.
\]
On the other hand, denoted by $h(\cdot,t)$ the PDF of $E_t$, we have
\begin{align*}
\mathbb{E}(|B_{E_t} - B_{E_s}|) &= \mathbb{E}(|\int_{E_s}^{E_t}dB_z|)\\
&=\int_0^{+\infty}\int_0^{+\infty} \mathbb{E}(|B_{r_1} - B_{r_2}|)h(r_1, t)h(r_2, s)dr_1 dr_2\\
&=\int_0^{+\infty}\int_0^{+\infty} \sqrt{r_1 - r_2} h(r_1, t)h(r_2, s)dr_1 dr_2\\
&=\mathbb{E}(\sqrt{E_t - E_s})\leq \sqrt{\mathbb{E}(E_t - E_s)}\\
&\leq C(\beta, 1)^{\frac{1}{2}}|t - s|^{\beta/2},
\end{align*}
which gives \eqref{est_m1}.
To prove \eqref{est_m2} we estimate
\begin{align*}
\int_{\R^d} |x|^2 dm_t(x) &= \E(|X_t|^2)\leq 2\E\left[|X_0|^2 + |\int_0^t b(s,X_s)dE_s|^2 + |\int_0^t \sigma(s, X_s)dB_{E_s}|^2\right]\\
&\leq 2\left(\E(|X_0|^2) + MC(\beta, 2)t^{2\beta} + M \int_0^{+\infty}\E(|B_s|^2)h(s,t)ds\right)\\
&\leq 2\left(\E(|X_0|^2) + MC(\beta, 2)t^{2\beta} + M\E(| E_t|)\right)\\
&\leq 2(\E(|X_0|^2) + MC(\beta, 2)t^{2\beta} + MC(\beta, 1)t^{\beta}.
\end{align*}
The thesis follows taking $c=2\max\{1, MC(\beta,2), MC(\beta, 1)\}$.
\end{proof}

\section{The fractional HJB equation}\label{sec_HJB}
In this section we introduce an optimal control problem for a class of time-changed diffusion processes
and we deduce, at a formal level, the corresponding dynamic programming equation, which turns
out to be a time-fractional HJB equation. \\
Let $(\Omega, \mathcal{F},\mathcal{F^t}, \mathbb{P})$ be a filtered probability space  and let $(B_t)_{t\ge 0} $ be a Brownian motion in $\R^d$   $\mathcal{F^t}$ adapted. \\
Fixed $T>0$,   consider  the controlled process $(X_s)_{s\geq t}$   given by the solution of the    time-changed  stochastic differential equation
\begin{equation}\label{stabledp}
\left\{
\begin{array}{l}
dX_s = f(s, X_s, u_s)dE_s +   \sqrt{2 \nu}\,dB_{E_s},\qquad s \in (t,T]\\
X_t=x.
\end{array}
\right.
\end{equation}
In \eqref{stabledp}, $E_s$ is a non increasing process, defined by
\begin{equation}\label{time_scale}
 E_s := T - \bar E_{T-s}
\end{equation}
with $\bar E_s$ the inverse of a $\beta-$stable subordinator   such  that $\bar E_0=0$, see \eqref{subor}, and  $\nu$ is a positive constant. The control law
$(u_s)_{s \geq 0}$ belongs to $\mathcal{U}$,   the  class   of progressive measurable processes taking values in the compact metric space $U$.
Observe that if the process $(Y_s)_{s\ge t}$ is the controlled diffusion process given by
\begin{equation}\label{dp}
dY_s = f(D_s,Y_s, \bar u_s)ds + \nu dB_s,
\end{equation}
then $X_s= Y_{E_s}$ is a solution of \eqref{stabledp} for  the time-changed control law  $u_s = \bar u_{E_s}$.\\
Given    $L: [0,T]\times \R^d \times U \to \R$  and   $g: \mathbb{R}^d \to \R$ representing respectively
the running cost and the  terminal cost, we consider  the cost functional
\begin{equation}\label{Jcost}
J(t,x,u) := \E_{x,t}\left[\int_t^T L(s, X_s, u_s)dE_s + g(X_T)   \right],
\end{equation}
where $X_s$ satisfies \eqref{stabledp}.
Note that the time in the cost functional  \eqref{Jcost} is rescaled according to the process $E_t$ such that $E_T=T$. Indeed, as   we explained in the previous section, $E_s$ is a change of time which represents an inner time scale for the process $X_s$ and also the cost functional
is evaluated according to this scale. Moroever,  the agent knows the final cost $g$ at time $T$
and,  accordingly to this datum, computes the optimal strategy backward in time (see also Remark \ref{rem_time}).\\
Define   the   value function $v: [0,T] \times \R^d \to \R$ by
\begin{equation}\label{value}
v(t,x) = \inf_{u \in \mathcal{U}}J(t,x,u).
\end{equation}
\begin{proposition}
Let $(t,x) \in [0,T) \times \R^d$ be given. Then, for every stopping time $\theta$ valued in $[t,T],$ we have
\begin{equation}\label{dpp}
v(t,x) = \inf_{u \in \mathcal{U}}\mathbb{E}_{x,t}\left[\int_t^\theta L(s,X_s, u_s)dE_s + v(\theta, X_\theta)\right].
\end{equation}
\end{proposition}
The proof   is based on standard arguments   in control theory (see \cite[Thm.3.3]{yz}). \\
Define the Hamiltonian $\cH: [0,T] \times \R^d \times \R^d   \to \R$  by
\begin{equation}\label{hamiltonian}
\cH(t,x,p ) := \sup_{u \in U} \left\{ -  f(t,x,u)\cdot  p - L(t,x,u)\right\}
\end{equation}
and   consider
the fractional HJB equation
\begin{equation}\label{HJBcap}
   \fdT v -\nu \Delta v +\cH(t,x,D v) = 0,\qquad(t,x) \in (0,T)\times\R^d
    \end{equation}
where $\fdT$ is the backward Caputo derivative defined in \eqref{bC}. We prove that the value function can be characterized as a solution of \eqref{HJBcap}.
We need a preliminary lemma.
\begin{lemma}\label{lem:density}
The function $h(\cdot, s)$, for $r\in [s,T]$, is the PDF of the process $E_s$ defined in \eqref{time_scale} iff it is a weak solution
of
\begin{equation}\label{B3}
\partial_s  h (r,s)= -D_{[s,T)}^{1-\beta}[\partial_r  h(r,s)] - \delta_T(r)\delta_T(s).
\end{equation}
\end{lemma}
\begin{proof}
Recalling that by definition $E_s= T - \bar E_{T-s}$, since the PDF $\bar h(r,s)$ of the process $\bar E_s$
satisfies (see \cite{bar,mst})
\[
\partial_{(0,s]}^{\beta}  \bar h(r,s) =- \partial_r  \bar h(r,s),
\]
it follows that
\[
\partial_{[s,T)}^{\beta}   h(r,s) = \partial_r  h(r,s).
\]
By \eqref{B1},  the previous equality can be equivalently  rewritten also as
\begin{equation}\label{B2}
D_{[s,T)}^\b  h(r,s) = \partial_r h (r,s) + \frac{(T-s)^{-\b}}{\G(1-\b)}\delta_T(r).
\end{equation}
We prove that  \eqref{B3} implies \eqref{B2}. Integrating \eqref{B3} from $s$ to $T$ we get
$$ h(r, T^-) -  h(r,s) = - I_{[s,T)}^{\beta} \partial_r  h(r,s) -\delta_T(r)H(T-s),$$
where $H$ is the Heaviside function. Observe that $ h(r, T^-) = h(T-r, 0^+) = 0$ (see \cite{mst} for more details). Then, applying $I_{[s,T)}^{1-\beta}$ on both sides, we get
$$-I_{[s,T)}^{1-\beta} h(r,s) = -\int_t^{T}\partial_r  h(r, \tau)d\tau - \delta_T(r)\frac{(T-s)^{1-\b}}{\G(2 - \b)}.$$
Taking the derivative $\frac{\partial}{\partial s}$, we finally get \eqref{B2}.\\
A similar computation shows that \eqref{B2} implies \eqref{B3}: as in \cite{mst}, applying the    fractional derivative in the sense of distribution to \eqref{B2} and using the identity  $D_{[s,T)}^{1-\beta}\partial_{[s,T)}^{\beta}\cdot  = \partial_s \cdot $, we obtain \eqref{B3}.
\end{proof}
\begin{proposition}\label{thm1}
Assume that   $f, L: \R^d \times [0,T] \times U \to     \R^d$,
are    continuous functions  satisfying
\begin{align}
     |f(t,x,u) - f(t,y,u)| +|L(t,x,u) - L(t,y,u)| &\leq  C |x - y| & \forall t \in [0,T],\, x,y\in\R^d, \,u\in U,\label{hype00c}\\
     |f(t,x,u)|+|L(t,x,u)| & \leq   M & \forall x \in\R^d, u\in U,\label{hype11c}
\end{align}
for some positive  constant $C, M >0$ and that  the value function $v$ defined in \eqref{value} is smooth. Then it is a classical solution of
\eqref{HJBcap}.
\end{proposition}
\begin{proof}
Given $(t,x) \in  [0,T)\times\R^d$ and $u\in U$, consider the constant process $u_t\equiv u$ for all $t \geq 0$ and let $X_t$
 be the corresponding solution of \eqref{stabledp}. For $\eta>0$, define the stopping time
$$\theta_\eta := \inf\{s>t : (s - t, X_s - x) \not\in [0,\eta)\times B_\eta \},$$
where $B_\eta$ is the ball in $\R^d$ with radius $\eta$. From the dynamic programming principle \eqref{dpp}, it follows that
\begin{equation}\label{f0}
\E_{x,t}\left[ v(t, X_t) - v(\theta_\eta, X_{\theta_\eta} )  \right] \leq \E_{x,t}\left[\int_t^{\theta_\eta}L(s, X_s, u)dE_s  \right].
\end{equation}
The right hand side term in \eqref{f0} can be rewritten in the following way
\begin{equation}\label{f1}
\begin{split}
\E_{x,t}\left[\int_t^{\theta_\eta}L(s, X_s, u)dE_s\right]  = \E_{x,t}\left[\int_{E_t}^{E_{\theta_\eta}}L(D_z, Y_z, u)dz\right]\\
 =- \int_{-\infty}^T\E_{x,t}\left[\left(\int_r^T L(D_z, Y_z, u)dz\right)(h(r,\theta_\eta) - h(r,t))\right]dr,
 \end{split}
\end{equation}
where $D_z$ is the process such that $E_{D_z} = z$, $h(\cdot, t)$ is the probability density function of $E_t$  and $Y_t$ is given by \eqref{dp}.
For a fixed $u\in U$   consider  the linear second order operator
\begin{equation*}
\mathcal{L}_u v(t,x):=     \nu     \Delta   v(t,x) +f(t,x,u )  D  v(t,x).
\end{equation*}
By Ito's formula (see \cite{kobayashi})
\begin{align*}
&\E_{x,t}\left[ v(\theta_\eta, X_{\theta_\eta}) - v(t, X_t) \right] =\E_{x,t}\left[\int_t^{\theta_\eta}d v(s, X_s) \right]\\
 &= \E_{x,t}\left[\int_t^{\theta_\eta} \partial_s v(s, X_s)ds + \int_t^{\theta_\eta}D(v(s, X_s))dX_s + \frac{1}{2}\int_{t}^{\theta_\eta}D^2v(s, X_s)d\langle X_s\rangle \right]\\
&=\E_{x,t}\left[\int_t^{\theta_\eta} \partial_s v(s, X_s)ds \right.+ \int_{t}^{\theta_\eta}\left(f(s,X_s,u)\cdot D v(s, X_s) +\nu \Delta  v(s, X_s) \right)dE_s \\
 &\left.+ \int_t^{\theta_\eta} \sqrt{2\nu}D v(s, X_s)dB_{E_s}\right] = \E_{x,t}\left[\int_t^{\theta_\eta} \partial_t v(s, X_s)ds + \int_{t}^{\theta_\eta}\mathcal{L}_u v(s, X_s)dE_s\right]\\
&=\E_{x,t}\left[\int_t^{\theta_\eta} \partial_t v(s, X_s)ds\right] -\int_{-\infty}^T\E_{x,t}\left[\left(\int_r^T \mathcal{L}_u v(D_z, Y_z)dz\right)(h(r,\theta_\eta) - h(r,t))\right]dr .
\end{align*}
Substituting the previous identity and \eqref{f1} in \eqref{f0} and dividing by $\eta$ on  both sides we get
\begin{align*}
 - \E_{x,t}\left[\frac{1}{\eta}\int_t^{\theta_\eta} \partial_t v(s, X_s)ds\right] &+ \int_{-\infty}^T\E_{x,t}\left[\left(\int_r^T \mathcal{L}_u v(D_z, Y_z)dz\right)\frac{h(r,\theta_\eta) - h(r,t)}{\eta}\right]dr \leq\\&\leq- \int_{-\infty}^T\E_{x,t}\left[\left(\int_r^T L(D_z, Y_z, u)dz\right)\frac{h(r,\theta_\eta) - h(r,t)}{\eta}\right]dr.
\end{align*}
Sending  $\eta\to 0^+$, since $\theta_\eta \to t$, by  the dominated convergence theorem we have
\begin{align*}
\E_{x,t}\left[ \frac{1}{\eta} \int_t^{\theta_\eta}\partial_t v(s,X_s)ds \right] \underset{h \to  0^+}{ \longrightarrow} \E_{x,t}\left[\partial_t v(t, X_t)  \right] = \partial_t v(t, x);\\
\int_{-\infty}^T\E_{x,t}\left[\left(\int_r^T (\mathcal{L}_u v(D_z, Y_z) + L(D_z, Y_z, u))dz\right)\frac{h(r,\theta_\eta) - h(r,t)}{\eta}\right]dr  \underset{\eta \to  0^+}{ \longrightarrow}\\
\int_{-\infty}^T\E_{x,t}\left[\left(\int_r^T (\mathcal{L}_u v(D_z, Y_z) + L(D_z, Y_z, u))dz\right) \partial_t h(r,t)\right]dr
\end{align*}
Set $\Phi(r) = \int_r^T (\mathcal{L}_u v(D_z, Y_z) + L( D_z,Y_z, u))dz$. Then by \eqref{B3}
\begin{align*}
\int_{-\infty}^T&\E_{x,t}\left[\Phi(r) \partial_r h(r,t)  \right]dr \\
&=  -\int_{-\infty}^T\E_{x,t}\left[\Phi(r)\fDT\partial_t h(r,t)   \right]dr - \E_{x,t}[\Phi(T)\delta_T(t) ]\\
&= -\int_{-\infty}^T\E_{x,t}\left[\fDT\left(\Phi(r)\partial_t h(r,t)  \right) \right]dr\\
&=- \E_{x,t}\left[ \fDT\left(\Phi(r)h(r,t)|_{-\infty}^T - \int_{-\infty}^T\partial_r \Phi(r)h(r,t)dr \right)  \right]\\
&= \E_{x,t}\left[\fDT\int_{-\infty}^T(L(D_r, Y_r, u) + \mathcal{L}_u v(D_r, Y_r))h(r,t)dr  \right]\\
&=\E_{x,t}\left[\fDT[ L(t, X_t, u) + \mathcal{L}_u v(t, X_t)]   \right]\\
&=\fDT[L(t,x,u) + \mathcal{L}_u v(t,x)].
\end{align*}
Then we have
\begin{equation}\label{bb}
 -\partial_t v(t,x) + \fDT[L(t,x,u) + \mathcal{L}_u v(t,x)] \leq 0, \qquad \forall (t,x)\in [0,T)\times \R^d.
\end{equation}
Set $\cF (t,x,u):= L(t,x,u) + \mathcal{L}_u v(t,x)$ and $v_T(x):=v(T,x)$. Recalling that
by definition $\fDT\, \cdot= -\frac{d}{dt}\left[\fIT\,\cdot\right]$, we can rewrite
\eqref{bb} as
$$-\partial_t v -  \frac{d}{dt}\left[ \fIT \cF (t,x,u)\right] \leq 0.$$
Applying $I_{[t,T)}^{1-\beta}$, we get
\begin{align*}
  \fdT v - I_{[t,T)}^{1-\beta}\fDT\cF(t,x,u) \leq 0.
\end{align*}
Moreover by \eqref{B1}
\begin{align}
I_{[t,T)}^{1-\beta}\fDT\cF(t,x,u) &=  I^{1-\beta}_{[t,T)}\left(\partial_{[t,T)}^{1-\beta}\cF(t,x,u) + \frac{(T-t)^{\beta - 1}}{\G(\beta)}\cF(T,x,u)\right)\nonumber\\
&= - I^{1-\beta}_{[t,T)}I^{\beta}_{[t,T)}\frac{d}{dt}\cF(t,x,u) + \frac{\cF(T,x,u)}{\G(\b)}I^{1-\beta}_{[t,T)}(T-t)^{\beta - 1}\label{calcoli1}\\
&= - \int_t^T \partial_s \cF(s,x,u)ds + \cF(T,x,u)\frac{\G(\b)\G(1-\b)}{\G(\b)\G(1-\b)}\nonumber\\
&= - \cF(T,x,u) + \cF(t,x,u) + \cF(T,x,u) = \cF(t,x,u).\nonumber
\end{align}
Hence
$$ \partial_{(t,T]}^{\beta}v(t,x) - \cF(t,x,u) \leq 0.$$
Since the previous inequality holds for any $u\in U$, we finally obtain
\[ \fdT v -\nu \Delta v +\cH(t,x,D v) \leq 0.\]
To complete the proof,   consider $t_0,t\in (0,T)$ such that $0 \leq t_0 < t \leq T$ and  such that $\eta = t - t_0$   small
and, for any $\epsilon > 0$, consider an $\epsilon$-optimal control $u^\epsilon$  such that
$$\E_{x_0,t_0}\left[\int_{t_0}^{t} L(s,X_s,u^\epsilon_s)dE_s  + v(t, X_{t})\right]\leq v(t_0, X_{t_0}) + \epsilon \eta;$$
due to \eqref{hype00c} and \eqref{hype11c}, we can apply the previous argument even if $u^{\epsilon}_t$ is not constant; then, dividing for $\eta$, from the Ito's formula it follows, for $\eta \to 0^+$,
$$-\partial_t v(t,x) + \fDT[L(t,x,u^\epsilon) + \mathcal{L}_u v(t,x)] \geq -\epsilon,$$
and, applying the same argument used in \eqref{calcoli1}, we get the thesis for the arbitrariety of $\epsilon$.
\end{proof}
We now discuss  the  notion of solution for \eqref{HJBcap} introduced  in \cite{kv2}.
Consider the space $C^p_\infty(\R^d)$, $p\ge 0$   given by the  functions $f \in C^p(\R^d)$ such that $f$ and its derivates
up to order $p$  are rapidly decreasing   functions on $\R^d$. To introduce a notion of solution for \eqref{HJBcap}, we first
consider a linear equation of the form
\begin{equation}\label{linear}
\left\{
\begin{array}{l}
\fdT w(t,x)- \nu\Delta w(t,x)= \ell(t,x),\\
 w(T,x)=w_T(x)
\end{array}
\right.
\end{equation}
for  given continuous functions  $w_T: \R^d\to\R$, $\ell:(0,T)\times\R^d\to\R$. If $w_T\in C^0_\infty(\R^d)$, a solution of \eqref{linear} can be written in the integral form
\begin{align}\label{f}
w(t,x)=\int_{\mathbb{R}^{d}}S_{\beta,1}(t, x-y)w_T(y)dy + \int_{t}^{T} \int_{\mathbb{R}^{d}}G_{\beta}(s-t, x-y ) \ell (s,y)dy ds.
\end{align}
where $S_{\beta, 1}( t, x-x_{0})$ and $G_{\beta}( T-t, x-x_{0})$ are  the solutions of equation  \eqref{linear} with $w_T(x)= \delta(x-x_{0})$, $\ell(t,x)=0$ and, respectively,  $w_T(x)=0$, $\ell(t,x)=\delta(T- t, x-x_{0})$ (in the case   $\beta = 1$, $G_{\beta}$ and $S_{\beta,1}$ coincide).
Explicit formula for $G_{\beta}$ and $S_{\beta,1}$ can be obtained in terms of the Fourier transform of Mittag-Leffler functions
(see \cite{kv2} for details). By formula \eqref{f},
it is natural to introduce  the following notion of solution for \eqref{HJBcap}
\begin{definition}\label{mild}
Given a continuous function $g\in C^0_{\infty}(\mathbb{R}^{d})) $, we say that $v\in C^0([0,T], C^{1}_{\infty}(\mathbb{R}^{d}))$ is a mild solution of \eqref{HJBcap} satisfying the terminal condition $v(x,T)=g(x)$  if
\[
v(t,x)=\int_{\mathbb{R}^{d}}S_{\beta,1}(t, x-y)g(y)dy + \int_{t}^{T} \int_{\mathbb{R}^{d}}G_{\beta}( s-t, x-y)\cH(s,y, D v(s,y))dy ds.
\]
\end{definition}
We quote  from  \cite{kv2}  the following existence and uniqueness result
\begin{theorem}\label{thm_HJB}
Assume that
\begin{itemize}
\item $\cH(s,x,p)$ is Lipschitz in $p$ with   Lipschitz constant $L_{1}$, i.e.
\begin{equation}\label{H1}
|\cH(s,x, p) - \cH(s,x, q)| \le L_{1}|p-q|.
\end{equation}
\item $\cH(s,x,p)$ is Lipschitz in $x$ with  Lipschitz constant $L_{2}$, i.e.
\begin{equation}\label{H2}
|\cH(s, x_{1}, p) - \cH( s, x_{2}, p)| \le L_{2}|x_{1}-x_{2}|(1 + |p|).
\end{equation}
\item $|\cH(s,x,0)| \le M_1$, for a constant $M_1$ independent of $x$.
\item $g  \in C^{2}_{\infty}(\mathbb{R}^{d})$.
\end{itemize}
Then there exists a unique  mild solution $v$ of  \eqref{HJBcap}. Moreover
$v\in C^0([0,T], C^{2}_{\infty}(\mathbb{R}^{d}))$ and
\[
{\sup}_{x}|D^{2}v(t,x))| \le  C
\]
with $C$ depending on $L_1$, $L_2$, $M_1$.
\end{theorem}
\begin{remark}
Assume that   $f, L$ satisfiy \eqref{hype00c}-\eqref{hype11c} and  the final cost $g$ is in $C^2_\infty(\R^N)$. Then the Hamiltonian $\cH$
defined in \eqref{hamiltonian} satisfies the assumptions of the previous theorem and therefore there exists a unique
mild solution to \eqref{HJBcap}   satisfying the terminal condition $v(T,x)=g(x)$.
Moreover a straightforward application of the  Verification Theorem, see e.g. \cite[Thm. 5.1]{yz} allows to conclude
that the mild solution coincides with the value function defined in \eqref{value}.
\end{remark}
\begin{remark}\label{rem_time}
The choice of the integration with  respect to $E_t$ in \eqref{Jcost}  is justified   for modeling purposes, since we want to calculate the cost when the particle is effectively moving. Note that HJB equation so obtained coincides    with the one
studied in \cite{gn,kv1,kv2,ty}. However,   considering integration with respect to the  standard time in \eqref{Jcost}, i.e.
\begin{equation*}
J(t,x,u) := \mathbb{E}_{x,t}\left[\int_t^T L(s,X_s, u_s)ds + g(X_T)   \right]
\end{equation*}
 and  repeating  an   argument similar to the one in the proof of Prop. \ref{thm1}, we find out that in this case  the value function $v$ satisfies
 the equation
\begin{equation*}
\fdT v  -\nu\Delta v+ \sup_{u \in U} \left\{  -f(t,x,u)\cdot Dv  - \fIT L(t,x,u)\right\}=0.
\end{equation*}
\end{remark}

 \section{Fractional Mean Field Games}\label{sec_MFG}
In this section we introduce the MFG system and we prove existence and uniqueness of a classical solution
to the problem.\\
Consider a population of indistinguishable agents  distributed
at time $t = 0$ according to the   density function $m_0$.
Each agent moves with a dynamics given by the time-changed SDE
\[
dX_s = f(s, X_s, u_s)dE_s +  \sqrt{2\nu}dB_{E_s} \qquad s \in (t,T],
\]
and aims to minimize the pay-off functional
\begin{equation}\label{cost_interaction}
 \mJ(t,x)=\E_{x,t}\left[\int_t^T [L(s,X_s, u_s)+G(X_s,m)]dE_s+g(X_T)\right],
\end{equation}
where   the additional term $G$ represents  a cost  depending on the distribution of the population at time $s$.
Given a distribution $m \in \cM^+([0,T]\times\R^d)$, by  Proposition \ref{thm1}  the
backward Hamilton-Jacobi equation associated to the previous control problem is
\begin{equation}\label{eqn:one}
\begin{split}
\fdT v(t,x)  -\nu \Delta v+\cH(t,x, D  v )-  G(x, m ) =0,\hspace{3 mm}(t,x) \in (0,T) \times \R^d,\\
\end{split}
\end{equation}
with the terminal condition $v(T,x) = g(x)$, where $\cH$ is defined as in \eqref{hamiltonian}.  \\
We  show that the time-fractional FP equation governing the evolution of the distribution $m$
 can be obtained by a standard duality argument in MFG theory.
 \begin{proposition}\label{prop_equiv}
The equation \eqref{eqn:one} is equivalent to the equation
     \begin{equation}\label{HJBfr}
-\partial_t v(t,x) + \fDT[-\nu \Delta v +\cH(t,x, D  v )-G(x,m)] = 0 \qquad (t,x) \in [0,T)\times\R^d,
\end{equation}
 where $\fDT$ denotes the backward Riemann-Liouville derivative \eqref{bRL}.
\end{proposition}
\begin{proof}
We get the thesis applying $D^{1-\beta}_{(t,T]}$ to \eqref{eqn:one} and recalling that $D^{1-\beta}_{(t,T]}\partial_{(t,T]}^{\beta}\cdot = - \partial_t\cdot.$
\end{proof}
For $\epsilon > 0$ and  $g_w \in C^2_{\infty}(\R^d)$,   write the  solution of \eqref{eqn:one} with terminal  data $g(x) + \epsilon g_w(x)$ as $v + \epsilon w$. Hence
$$
-\partial_t v - \epsilon \partial_t w + D_{[t,T)}^{1-\beta}[ - \nu \Delta v - \epsilon \nu \Delta w+\cH(t,x, D v + \epsilon D w) -G(x,m)] =0.
$$
By Taylor's expansion, the Hamiltonian term can be rewritten as
$$
\cH(t,x, D v + \epsilon D w) = \cH(t,x, D v) + \epsilon D_p\cH(t,x, D v)D w + o(\epsilon^2).
$$
Substituting in \eqref{eqn:one} and isolating the terms with the same  order in $\epsilon$, we  get the following
equation for $w$
\begin{equation}\label{eqn:two}
-\partial_t w + \fDT\left[ - \nu \Delta w+ D_p\cH(t,x, D v)D w \right]= 0, \quad(t,x) \in [0,T]\times\R^d\\
\end{equation}
Assume that $w$ is a weak solution of \eqref{eqn:two}. Integrating  with respect to a test function  $m \in C^{1}_c((0,T); C^2_{\infty}(\R^d))$,
we have
\begin{equation}\label{eqn:three}
\int_0^T \int_{\R^d} \left[-\partial_t w + \fDT\left( - \nu \Delta w+D_p\cH(t,x, D v)D w \right)\right] m(x,t)dxdt= 0.
\end{equation}
Then, taking into account \eqref{bypart2}, we get that \eqref{eqn:three} is equivalent to
\begin{equation}\label{id11}
 \int_0^T \int_{\R^d} w(x,t)\left[\partial_t m -\nu \Delta(\fDz m)-  \diver\left(D_p \cH(t,x, D v)\fDz m\right)\right] dxdt=0
 \end{equation}
By \eqref{id11}, we deduce the   time-fractional FP problem
\[
\left\{
\begin{array}{ll}
\partial_t m(t,x) = \cA \left[\fDz m(t,x)\right] \quad& (t,x) \in (0,T)\times \R^d,\\[4pt]
m(0,x) = m_0(x),&x\in\R^d,
\end{array}
\right.
\]
where
\[
\cA \cdot  =   \nu \Delta \cdot+ \diver (D_p \cH(t,x, D  v)\cdot)
\]
and $-D_p \cH(t,x, D  v)$ is  the optimal control obtained by \eqref{eqn:one},
is the adjoint of the linearized of the HJB equation, as in the standard MFG theory.\\
We are now ready to formulate the fractional Mean Field Game system as
\begin{equation}\label{MFGfr}
\left\{
\begin{array}{lll}
-\partial_tv+ \fDT[-\nu \Delta v  + \cH(t,x, D  v) -    G(x, m)]=0,& (t,x) \in (0,T) \times \R^d\\[4pt]
\partial_t m -  [\nu \Delta \cdot+ \diver (D_p \cH(t,x, D  v)\cdot)] (\fDz m) = 0,\\
m(0,x) = m_0(x) ,\quad v(T,x) = g(x).
\end{array}
\right.\end{equation}
\begin{remark}
We stress that \eqref{MFGfr} is  the correct form of MFG system preserving the duality relation between the two equations. Indeed,
if we consider the HJB equation \eqref{eqn:one} in place of the HJB equation \eqref{HJBfr} in the system \eqref{MFGfr}, a computation similar to
\eqref{eqn:three}-\eqref{id11} gives the FP equation
\[ D_{(0,t]}^{\beta} m -  \nu \Delta m - \diver (D_p \cH(t,x, D  v)m)  = 0\]
which is not well posed since the corresponding solution $m$ is not a normalized, non-negative distribution probability function (\cite{h}).
\end{remark}

\subsection{Uniqueness}
We assume that   $G: \R^d \times \cM^+(\R^d \times [0,T]) \to \R$ is a continuous function whose backward fractional  Riemann-Liouville derivative of order $1-\beta$ is monotone with respect to the $m-$variable, i.e.
\begin{equation}\label{monotonicity}
\int_0^T \int_{\R^d} (m_1 - m_2) \fDT\left(  G(x, m_1)   -   G(x, m_2)\right)dxdt> 0,
\end{equation}
for every $m_1, m_2 \in \cM^+([0,T]\times\R^d)$.  \\
For example, the previous condition is satisfied if $G(m) = I_{[t,T)}^{1-\beta}\g(m)$ with  $\g$  an  increasing function.
Moreover, if the interaction cost is computed with respect to the external time scale, hence  the term $\E_{x,t}\left[\int_t^T G(m_s) dE_s \right]$
in \eqref{cost_interaction} is replaced by $\E_{x,t}\left[\int_t^T G(m_s) ds\right]$, then \eqref{monotonicity} is satisfied if $G(m)$ is increasing in $m$,
as in the classical monotonicity condition in \cite{ll}.
\begin{theorem}
There exists a unique classical solution to the MFG system \eqref{MFGfr}.
\end{theorem}
\begin{proof}
We apply  a standard argument  in   MFG theory.  We assume that there exists two solutions $(v_1,m_1)$ and $(v_2,m_2)$ of \eqref{MFGfr}. We set $\bar v=v_1-v_2$,
$\bar m=m_1-m_2$ and we write the equations for $\bar v$, $\bar m$
\[\left\{\begin{array}{ll}
 \,-\partial_t \bar v  +  \fDT[- \nu \Delta \bar v +\cH(t,x, D v_1) - \cH(t,x, D v_2)-(G(x, m_1) - G(x,m_2))]=0 \\[2pt]
\partial_t \bar m - \nu \Delta(\fDz \bar m) - \diver(D_p \cH(t,x, D v_1)\fDz m_1)  + \diver(D_p \cH(t,x, D v_2)\fDz m_2) =0 \\[2pt]
       \bar m(0,x)=0,\,\bar v(t,x)=0
\end{array}\right.\]
Multiplying the equation for $\bar m$ by $\bar v$ and integrating  we get
\begin{equation}\label{U4}
\begin{split}
 & \int_0^T\int_{\R^d} \left[\bar v \partial_t \bar m + \nu D \bar v\cdot D \left(\fDz \bar m\right) \right. \\&\left.+ D \bar v \cdot\left(D_p \cH(t,x, D v_1) \fDz
  m_1 - D_p \cH(t,x, D v_2) \fDz m_2 \right)\right]dxdt=0.
  \end{split}
\end{equation}
Multiplying the equation for $\bar v$  by $\bar m$ and performing a computation similar to the one in \eqref{conto}, we have
\begin{equation}\label{U5}
\begin{split}
 &\int_0^T\int_{\R^d}\Big[\bar v \partial_t \bar m + \fDz\bar m \left(\cH(t,x, D v_1) - \cH(t,x, D v_2) \right) +\\&+  \nu D \left(\fDz \bar m\right)\cdot D \bar v - \bar m \,  \fDT\big(G(x, m_1) - G(x,m_2)\big)\big]dx dt=0.
\end{split}
\end{equation}
Subtracting \eqref{U4} to \eqref{U5}, we get
\begin{align*}
    &\int_0^T\int_{\R^d} (m_1-m_2)\, \fDT (G(x,m_1) -G(x,m_2))dxdt +\\
    & \int_0^T\int_{\R^d} \fDz m_1 \big(\cH (t, x,D v_2) - \cH (t, x,Dv_1) -D_p \cH (t, x,D v_1)D( v_2-v_1)\big)dx dt+\\
     &\int_0^T\int_{\R^d} \fDz m_2 \big(\cH (t, x,D v_1) - \cH (t, x,Dv_2) -D_p \cH (t, x,D v_2)D( v_1-v_2)\big)dxdt=0.
    \end{align*}
Since  each of the three terms in the previous identity is nonnegative, in view of assumption \eqref{monotonicity} it follows that $m_1=m_2$. By the uniqueness of the solution to  \eqref{HJBfr}, we finally get $v_1=v_2$.
\end{proof}
\subsection{Existence}
We now  prove  existence of a  classical solution to  the MFG system \eqref{MFGfr}. In this section we assume that  the Hamiltonian
$\cH\in C^1([0,T]\times\R^d\times\R^d)$ and satisfies  \eqref{H1}-\eqref{H2}, $g \in C^2_{\infty}(\R^d)$, $m_0\in \cP_1(\R^d)$ is such that $\int_{\R^d}|x|^2 dm_0<\infty$  and $G$  is uniformly bounded and Lipschitz continuous, i.e.
\begin{align*}
& |G( x_1,m_1) - G( x_2, m_2)| \leq C_1[|x_1 - x_2| + d_1(m_1, m_2)], \quad&&\forall   (x_1, m_1),(x_2, m_2) \in \R^d\times\cM^+_1(\R^d),\\
&|G( x,m)| \leq C_2, &&\forall ( x, m) \in \R^d\times\cM^+_1(\R^d),
\end{align*}
for some positive constant $C_1, C_2 > 0.$
\begin{theorem}
There exists a solution $(v,m)  \in C([0,T]; C^1_\infty(\R^d))\times C([0,T];\cP_1(\R^d))$  to \eqref{MFGfr},
where $v$ is a mild solution of the HJB equation and $m$ is a weak solution   of the FP equation.
\end{theorem}
\begin{proof}
 We prove   existence of  a solution by  a fixed point  argument. Let $\cC$ be the subset of     $C([0,T], \cP_1(\R^d))$ given by
 distribution functions which are   $\frac{\beta}{2}-$Holder continuous with constant $M_1$ (to be fixed later) and such that
\begin{equation}\label{hypecomp}
\sup_{t \in [0,T]}\int_{\R^d}|x|^2 d\mu_t(x) \leq M_1.
\end{equation}
Since $\cP_1$ is a convex set, closed with respect to $d_1$, then  $\cC$ is convex and closed  with respect to the distance
 $\sup_{t \in [0,T]}d_1(\mu_t, \nu_t)$. It is also compact due to \eqref{hypecomp}.
We define a map  $\Phi: \cC \to  C([0,T], \cP_1(\R^d))$ in the following way:  \\
\emph{(i)} Given   $\mu \in \cC$,  consider the HJB equation
\[
\left\{
\begin{array}{ll}
-\partial_t v(t,x) + \fDT(- \nu \Delta v+ \cH(t,x, D v)-G(x, \mu) ) =0,& (t,x) \in (0,T) \times \R^d\\
v(T,x) = g(x).
\end{array}
\right.
\]
By Theorem \ref{thm_HJB} and Proposition \ref{prop_equiv}, there exist a unique mild solution $v=v(\mu) \in C^0([0,T]; C^1_\infty(\R^d))$.\\
 \emph{(ii)}
Given $v$ by the previous step, consider the
time-fractional    FP equation
\[
\left\{
\begin{array}{lll}
\partial_t m -   [ \nu \Delta\cdot+\diver(D_p \cH(t,x, D v)\cdot]\fDz m  = 0,\\
m(0,x) = m_0(x).
\end{array}
\right.
\]
By Theorem \ref{thm_FFP1}, there exists a unique weak solution $m \in C([0,T], \cP_1(\R^d))$.\\
Hence the map $m:=\Phi(\mu)$ defined by steps  \emph{(i)}-\emph{(ii)}  is well defined. Moreover, for
$$M_1 := \max\{c(\beta,M)(\E(|X_0|^2) + T^\beta + T^{\beta/2}), C(\beta, M)\},$$
see \eqref{est_m1}-\eqref{est_m2},   we have that  $m \in \cC$, hence  $\Phi$ maps $ \cC$ into itself.
We show that the map $\Phi$ is continuous. Consider a sequence $\mu_n \in \cC$ converging  to some $\mu\in \cC$ and let $(v_n, m_n)$, $(v, m)$ the corresponding  functions defined in steps \emph{(i)}-\emph{(ii)}.
\\
In  \cite[Thm.11]{kv1}, it is shown   that there exists a positive constant $C_0$    such that  the solution of the HJB equation is bounded in $C^2(\R^d)$ for every $t \in [0,T]$. Since the constant $C_0$ depends only on the bounds \eqref{H1}-\eqref{H2}, it follows that
$$\sup_{t \in [0,T]}\|v_n\|_{C^2(\R^d)} \leq C_0,$$
uniformly in $n\in \N$.
Hence we conclude that, up to a subsequence,  $v_n$ and $D v_n$ locally uniformly converge, respectively, to $v$ and $D v$.\\
It is easily seen that  any converging  subsequence   of the relatively compact sequence $m_n$  is  a weak solution of the FP equation associated to $v$. Since the  solution of this equation  is unique, we get that all the sequence $m_n$ converges to $m$ and therefore  $\Phi$ is continuous.\\
 By the Schauder fixed point Theorem, the map $\Phi$ admits a fixed point $m=\Phi(m)$ in $\cC$. It follows that
the corresponding couple $(v,m)$ defined in steps  \emph{(i)}-\emph{(ii)} is a solution of \eqref{MFGfr}.
\end{proof}
\begin{remark}
We consider    steady solution of \eqref{MFGfr}, i.e. solutions such that  $v(x,t) = v(x,T)$, $m(x,t) = m(x,0)$,  for all $t \in [0,T]$. If $(v,m)$ is a steady solution, then  \eqref{MFGfr} reduces to
 \begin{equation}\label{eqn:steadysys}
 \begin{split}
 \fDT(- \nu \Delta v+\cH (t,x, D v)  - G(x, m)) = 0, \\
 [\nu \Delta\cdot+ \diver(D_p \cH(t,x, D v)\cdot ] (\fDz m)   = 0.
 \end{split}
 \end{equation}
 The first equation in \eqref{eqn:steadysys} is equivalent to
 $$- \nu \Delta v+\cH(t,x, D v)  - G(x, m) = c(x) (T-t)^{1-\b}.$$
If $\cH$ does not depends on $t$, then the previous equation is satisfied   if and only if $c(x)\equiv 0$ and
 \begin{equation*}
- \nu \Delta v + \cH(x, D v) - G(x, m)=0, \quad x \in \R^d.
 \end{equation*}
 On the other hand, due to the independence of $m$ on the $t-$variable, it can be easily verified that
 $$D_{(0,t]}^{1-\beta} m = C_1(\beta)  t^{\beta - 1}m(x) ,$$
for some constant  $C_1 \not = 0$. Then, the second equation in \eqref{eqn:steadysys} is satisfied if
 \[
  \nu \Delta m+ \diver(D_p \cH( x, D v)m)  = 0.
 \]
We conclude that $(v,m)$ is a steady solution of \eqref{MFGfr} if it solves
 \[
 \left\{
 \begin{array}{ll}
 - \nu \Delta v + \cH(x, D v) = G(x, m), & x \in \R^d,\\
  \nu \Delta m + \diver(D_p \cH( x, D v)m)   = 0.
 \end{array}
 \right.
 \]
 Hence steady solutions of \eqref{MFGfr} coincide with the ones of the classical MFG system.
 \end{remark}



\end{document}